\documentclass[11pt]{article}
\usepackage[T1]{fontenc}
\usepackage[utf8]{inputenc}
\usepackage{amssymb}
\usepackage{amsmath}
\usepackage{enumerate}
\usepackage{amsmath,amssymb}
\usepackage{lmodern}
\usepackage{fourier, heuristica}
\usepackage{mathtools, nccmath}
\usepackage{mathtools}
\usepackage[english]{babel}
\usepackage{amsthm}
\usepackage{systeme}
\usepackage{geometry}[margin=1in]
\usepackage{authblk}
\newtheorem{theorem}{Theorem}[section]

\newtheorem{lemma}[theorem]{Lemma}
\newtheorem{definition}[theorem]{Definition}
\newtheorem{ques}[theorem]{Question}
\newtheorem{prop}[theorem]{Proposition}

\newtheorem*{theorem*}{Theorem}
\newtheorem*{remark}{Remark}
\newtheorem*{mydefinition*}{}
\usepackage{graphicx}
\usepackage{relsize}
\usepackage{tcolorbox}
\usepackage{tikz-cd}
\usepackage[all]{xy}
\usepackage{url}
\numberwithin{equation}{section}
\usepackage{eqparbox}

\usepackage{hyperref, cleveref}

\usepackage{hyperref}
\hypersetup{
    colorlinks=true,
    linkcolor=blue,
    filecolor=magenta,      
    urlcolor=cyan,
    pdftitle={Overleaf Example},
    pdfpagemode=FullScreen,
    }

\makeatletter
\newcommand{\address}[1]{\gdef\@address{#1}}
\newcommand{\email}[1]{\gdef\@email{\url{#1}}}
\newcommand{\@endstuff}{\par\vspace{\baselineskip}\noindent\small
\begin{tabular}{@{}l}\scshape\@address\\\textit{E-mail address:} \@email\end{tabular}}
\AtEndDocument{\@endstuff}
\makeatother
\title{Noncongruence cusp forms, Veech groups and enumeration}
\author{John Rached}
\date{}
\address{Department of Mathematics and Statistics,
Binghamton University. Binghamton, New York, USA}
\email{jabourached@math.binghmaton.edu}

\usepackage{amsmath,calligra,mathrsfs}

\DeclareMathOperator{\Hom}{\mathscr{H}\text{\kern -3pt {\calligra\large om}}\,}

\begin{document}

\maketitle

\begin{abstract}
We prove the existence of weight 2 modular forms for some Veech groups, whose Fourier coefficients at a cusp determine the counts of numbers of branched Galois covers of punctured elliptic curves. We further prove there exist cusp forms, counting ramified genus $2$ covers, that are not invariant under any congruence subgroup of $\mathrm{SL}(2,\mathbb{Z})$. The main results may be viewed as analogs of the construction of an elliptic curve obtained as the quotient of the Jacobian of a classical modular curve, whose point counts modulo primes determine the coefficients of a Hecke eigenform, but where the uniformizing groups may be noncongruence. 
\end{abstract}

\tableofcontents

\section{Introduction}

Each stratum of the space of translation surfaces contains infinitely many surfaces whose Veech groups are $\textit{totally noncongruence}$ \cite{schlagepuchta2018}; in each genus, one can find infinitely many surfaces whose Veech groups are finite-index in $\mathrm{SL}(2,\mathbb{Z})$, but whose projections onto $\mathrm{SL}(2,\mathbb{Z}/N\mathbb{Z})$ are surjective for every $N \in \mathbb{Z}$. On the other hand, even generalizing to the case where Veech groups are non-arithmetic Fuchsian, experimental evidence \cite{kumar2017algebraic} suggests Veech groups have a rich theory of modular forms. In this note, we focus on the noncongruence arithmetic situation, and for a quite special class of cases, show there is some analog for the role of Hecke eigenforms.  \\

We recall that, for congruence modular curves, a classical construction associates to a cuspidal weight 2 Hecke eigenform with $\mathbb{Q}$-coefficients an elliptic curve over $\mathbb{Q}$, such that the coefficients of the eigenform are determined by the point-counts of the reduction of the elliptic curve at good primes. The "inverse" construction, associating to any elliptic curve over $\mathbb{Q}$ a modular form whose coefficients are determined by the point-counts modulo primes \cite[\S 8.8]{diamond2005first}, is the celebrated modularity theorem. We provide a quick review below, to motivate our results.\\

A finite-index subgroup $\Gamma \leq \mathrm{SL}(2,\mathbb{Z})$ is $\textit{congruence}$ if it contains a principal congruence subgroup. This means that there is an integer $N \geq 1$ such that $\Gamma(N)\subseteq \Gamma$ where
\begin{displaymath}
\Gamma(N):= \textrm{ker}(\mathrm{SL}(2,\mathbb{Z}) \mapsto \mathrm{SL}(2,\mathbb{Z}/N\mathbb{Z}))
\end{displaymath}
is the principal congruence subgroup of level $N$. Since $\Gamma \leq \mathrm{SL}(2,\mathbb{Z})$ acts on the hyperbolic upper-half plane $\mathbb{H}$ by M{\"o}bius transformations, we may form the quotient $\mathbb{H}/\Gamma$, which, in general, can be endowed with the structure of a complex dimension one orbifold. There is a subgroup $\Gamma_0(N)$, containing $\Gamma(N)$ as a normal subgroup, such that the coarse scheme of $\overline{\mathbb{H}}/\Gamma_0(N)$ has a model over $\mathbb{Q}$ coming from a moduli interpretation; this is commonly denoted $X_0(N)$. For integers $k \geq 0$, we have $\textit{weight $k$ modular forms}$ on $X_0(N)$; the spaces of sections $H^0(X_0(N),\omega^{\otimes k})$ of a line bundle (due to presence of elliptic points in this case, must be in orbifold sense)  $\omega$ on $X_0(N)$.  From the analytic description of these sections, it is immediate that they satisfy a periodicity condition at cusps, so that for a weight $k$ modular form $f$, there is a $\textit{q-expansion}$ at a cusp of width $h \geq 1$ given by
\begin{displaymath}
f = \sum\limits_{n \geq 1} a(n) q^{n}, \hspace{.2cm} q = e^{2\pi iz/h}
\end{displaymath}
where $z$ is a local coordinate on $\overline{\mathbb{H}}/\Gamma_0(N)$. For good primes $p$, there are natural algebraic correspondences
\begin{displaymath}
X_0(pN) \mapsto X_0(N) \times X_0(N)
\end{displaymath}
so that functorial pull-push constructions (on the associated roof diagram of the correspondence) induce operators on various linear objects associated to $X_0(N)$ on spaces of modular forms of a given weight. The linear operators on the spaces of weight $k$ modular forms are the $\textit{Hecke operators}$ $T_p$. Given a weight $2$ modular form $f$ vanishing at all cusps, that is a simultaneous eigenvector for all Hecke operators, and whose $q$-expansion at a cusp has all coefficients $a(n)$ lying in $\mathbb{Q}$, the fact alluded in the first paragraph of the section says the following: there is an elliptic curve $E_f$ over $\mathbb{Q}$, such that

\begin{displaymath}
a(p) = 1 + p - \#E_f(\mathbb{F}_p).
\end{displaymath}
for primes $p$ not diving the conductor of $E_{f}$. \\

We say a finite index subgroup $\Gamma$ of $\mathrm{SL}(2,\mathbb{Z})$ is $\textit{noncongruence}$ if there is no integer $N \geq 0$ such that $\Gamma$ contains a principal congruence subgroup $\Gamma(N)$ as defined above (as mentioned above, it is totally noncongruence if the stronger condition holds, that the subgroup is proper and all canonical projections are surjective. Equivalently, it is a proper subgroup dense in the congruence completion $\mathrm{SL}(2,\hat{\mathbb{Z}})$). Despite the fact finite index subgroups of $\mathrm{SL}(2,\mathbb{Z})$ have $\mathrm{SL}(2,\mathbb{Q})$ as their commensurator in $\mathrm{SL}(2,\mathbb{R})$, which allows one to define Hecke operators on noncongruence $\Gamma$ just the same as for congruence subgroups, Serre-Burger prove this Hecke action is trivial on (genuine) noncongruence modular forms \cite{Berger1994}. One important consequence of the congruence theory of Hecke operators is that simultaneous eigenvectors whose Fourier expansions have $\mathbb{Q}$-coefficients, in fact have coefficients with \textit{bounded denominators} \cite[Theorem 3.5.2]{shimura1971arithmetic}; the expansions have integer coefficients, up to normalization. Breakthrough work of Calegari-Dimitrov-Tang \cite[Theorem 1.0.1]{CalegariDimitrovTang_UnboundedDenominators} confirms the conjecture of Atkin Swinnerton-Dyer that the bounded denominators property actually characterizes (genuine) noncongruence modular forms. In special settings, we show there yet exist weight 2 modular forms on potentially noncongruence modular curves whose $q$-expansion coefficients "count $\mathbb{F}_p$-points on a moduli space". \\

Let $G$ be a finite group, and let $\mathcal{M}(G)$ be the moduli stack of $G$-structures on elliptic curves over $\mathbb{Z}[1/|G|]$, constructed by Chen \cite{chen2018moduli} (Section \ref{modulistacks}). This space essentially \footnote{This is literally true when the center of $G$ is trivial; however, in general, $\mathbb{Q}$-points of $\mathcal{M}(G)$ for example, may not arise from $G$-Galois covers} parametrizes $G$-Galois covers of elliptic curves with ramification only possible over the origin. The connected components of the underlying coarse scheme of $\mathcal{M}(G)_{\mathbb{C}}$ are modular curves (Section \ref{TM_section}), for modular group a finite-index, possibly $\textit{totally noncongruence}$ subgroup of $\mathrm{SL}(2,\mathbb{Z})$. Further, for a cover $\pi:C \rightarrow E$ of an elliptic curve by a surface, with ramification only possibly over the origin $O$, if the unramified cover $C^{\circ} \rightarrow E^{\circ}$ where $E^{\circ} := E - \left\{O\right\}$ and $C^{\circ}:= C - \left\{\pi^{-1}(O)\right\}$ is Galois, the coarse scheme of the associated Teichm{\"u}ller curve (viewed as an orbifold quotient, see Appendix \ref{Appendix_A}) can be identified with the connected component $\mathcal{H}(\pi)_{\mathbb{C}}$ of the coarse scheme of $\mathcal{M}(G)_{\mathbb{C}}$, containing $\pi$ (Section \ref{TM_section}). Also, (Proposition \ref{compactification}) there is a containment $\mathcal{M}(G) \subset \overline{\mathcal{M}(G)}$ as a dense open substack of a smooth, proper Deligne-Mumford stack over $\mathbb{Z}[1/|G|]$. Let $\overline{\mathcal{H}}(\pi)$ denote the closure of $\mathcal{H}(\pi)$ in $\overline{\mathcal{M}(G)}$. As is conventional, we use subscripts to denote base changes, e.g. $\mathcal{M}(G)_{\mathbb{C}}$ for base change to $\mathbb{C}$. The following is our first main result.

\begin{theorem}
\label{main}
Let $\pi:C^{\circ} \rightarrow E^{\circ}$ be a degree $d$, unramified cover of an elliptic curve $E$ with origin removed. Then, the compactification of the coarse scheme of the modular curve associated to the Veech group of its Galois closure $\pi':C' \rightarrow E^{\circ}$ can be identified with a compactified component $\overline{\mathcal{H}}(\pi')_{\mathbb{C}}$ of the underlying coarse scheme of $\mathcal{M}(G)_{\mathbb{C}}$ where $G$ is a subgroup of the symmetric group $S_d$. If $\overline{\mathcal{H}}(\pi')$ is defined over $\mathbb{Q}$ and $p \nmid d!$, $\overline{\mathcal{H}}(\pi')$ has good reduction at $p$. If, in addition, $\overline{\mathcal{H}}(\pi')$ is of genus $1$, and has a $\mathbb{Q}$-rational point at a cusp $c$, there exists a weight 2 modular form $f$, for the Veech group, satisfying the following: the Fourier coefficients $a(n)$ of $f$ in its $q$-expansion at the cusp at $c$
\begin{displaymath}
f(q) = \sum_{n \geq 1}a(n)q^n,
\end{displaymath}
determine the cardinality of the \ $\mathbb{F}_p$-points, $\#\overline{\mathcal{H}}(\pi')(\mathbb{F}_p)$.
\label{main_theorem}
\end{theorem}

We also show that, for a moduli interpretation for Veech group quotients over $\mathbb{Q}$, closely related but different to that of a connected component of $\mathcal{M}(G)$, we can guarantee the modular form $f$ determining the point-counts is $\textit{not}$ a Hecke eigenform, and is in fact not invariant under $\textit{any congruence subgroup}$. For a nonsquare discriminant $D$, the Hilbert modular surface $X_D$ parametrizes principally polarized abelian surfaces equipped with real multiplication by the real quadratic order $\mathcal O_D$. The case of square discriminant $D=d^2$ is different: the associated quadratic algebra is split, with $\mathcal O_{d^2}\otimes_{\mathbb Z}\mathbb Q \cong \mathbb Q\times\mathbb Q$, rather than a real quadratic field. In this case, $X_{d^2}$ admits a birational moduli interpretation \cite[Section 4]{McMullen2007Dynamics}, \cite[Section 4]{Kani1994} in terms of primitive degree-$d$ maps $\phi:Y\to E$, where $Y$ is a smooth genus-two curve and $E$ is an elliptic curve. The induced map $\phi_*:\operatorname{Jac}(Y)\to E$ has connected kernel an elliptic curve, giving two elliptic curves naturally associated to the cover. Thus one obtains a map from $X_{d^2}$ to $\mathcal M_1\times\mathcal M_1$, recording these two elliptic curves. The locus $W_{d^2}$  where a ramification point of the degree-$d$ map is a Weierstrass point of $Y$ (or, equivalently, the locus of genus-two eigenforms with a double zero) is comprised of {Teichm\"{u}ller curves} (see Section \ref{Wohlfahrt}) as their irreducible components. For a primitive degree-$d$ map $f:Y\to E$ with $g(Y)=2$ and $g(E)=1$, Riemann–Hurwitz gives $2g(Y)-2=d(2g(E)-2)+\sum_{y\in Y}(e_y-1)$, so $2=\sum_{y\in Y}(e_y-1)$. Thus the cover has total ramification $2$, realized by two simple ramification points. The condition defining $W_{d^2}$ is then that one of these ramification points is a Weierstrass point of $Y$. The members of the family $W_{d^2}$ are referred to as $\textit{Weierstrass curves}$. The Fuchsian groups (Veech groups) uniformizing components of $W_{d^2}$ are noncongruence, finite index in $\mathrm{SL}(2,\mathbb{Z})$.
Mukamel \cite{mukamel2021polynomials} studies the image of $W_{d^2}$ under the map to the two $j$-invariants and obtains an explicit polynomial relation over $\mathbb{Z}$ whose zero locus is birational to $W_{d^2}$. We show

\begin{theorem}
\label{maintheorem2}
There exists a noncongruence cusp form $f$, invariant a Veech group for a Weierstrass curve, and not invariant under any congruence subgroup of $\mathrm{SL}(2,\mathbb{Z})$ such that its Fourier expansion at a $\mathbb{Q}$-rational cusp $c$
\begin{displaymath}
f(q) = \sum\limits_{n \geq 1}a(n) q^n
\end{displaymath}
determine, for all but finitely many primes $p$, the number of primitive, degree-$d$, genus $2$ covers of elliptic curves $\phi:Y \rightarrow E$ over $\mathbb{F}_p$ with two simple ramification points (up to an additive error independent of $p$).
\end{theorem}

In fact, this is the curve $W_{64}$ and Mukamel \cite[Theorem 3.1]{mukamel2021polynomials} has determined an explicit elliptic curve
\begin{displaymath}
y^2 = x^3 - 63x + 162
\end{displaymath}
birational to $W_{64}$.
\begin{remark}
When center of $G$, $Z(G)$, is trivial, Theorem \ref{main_theorem} is an enumeration statement. Specifically, for $p \nmid d!$, since the coefficients $a(n)$ determine $\#\overline{\mathcal{H}}(\pi')(\mathbb{F}_p)$, this means $a(n)$ determine $\#\mathcal{H}(\pi')(\mathbb{F}_p)$ up to an additive error, independent of $p$, coming from the boundary contributions $\partial \mathcal{H}(\pi') = \overline{\mathcal{H}}(\pi')\setminus \mathcal{H}(\pi') \subset \overline{\mathcal{M}(G)}$.
\end{remark}

\subsection{Overview of argument}

First, we realize the Teichm{\"u}ller curve associated to the Galois closure of $\pi:C^{\circ} \rightarrow E^{\circ}$ as a connected component of $\mathcal{M}(G)_{\mathbb{C}}$. This is already sketched in \cite[Section 4.2]{Chen2025NoncongruenceHurwitz}. In particular, we show in Appendix \ref{Appendix_A} that Teichm{\"u}ller uniformization affords an equivalence of categories (commuting with the projection functors) from the the orbifold quotient of the hyperbolic plane modulo the Veech group, and the analytification of a connected component of $\mathcal{M}(G)_{\mathbb{C}}$. \\

Upon making this identification, and to prove Theorem \ref{main_theorem}, it suffices to show if $\overline{\mathcal{H}}(\pi')_{\mathbb{C}}$ is the compactification of a genus $1$ connected component of the underlying coarse scheme of $\mathcal{M}(G)_{\mathbb{C}}$, defined over $\mathbb{Q}$, and containing a $\mathbb{Q}$-rational point at cusp, then the conclusion of Theorem \ref{main} holds. This is because the compactification of $\overline{\mathcal{M}(G)}$ is smooth over $\mathbb{Z}[1/|G|]$ and thus has good reduction at $p \nmid |G|$. For this, we will use Selberg's bound on the Fourier coefficients of modular forms, valid for \textit{any} finite-index subgroup of $\mathrm{SL}(2,\mathbb{Z})$ \cite{selberg1965}, combined with a three-term $p$-adic recursion relation for Fourier coefficients of holomorphic differentials on elliptic curves (\cite{katz1981crystalline}, \cite{Scholl1985}, \cite{Kibelbek2014ASD}, \cite{atkin1971modular}), and a simple Hasse bound on $\mathbb{F}_p$-points on elliptic curves. Similar arguments, where these recursion relations and Fourier bounds are both leveraged, appear for closely related purposes to ours, for example, \cite[Section 3]{li2012fourier}. \\

To obtain Theorem \ref{maintheorem2}, the same strategy of Selberg bound, recursion relation and Hasse bound works, but we use a dimension count (Lemma \ref{existence_noncongruence}) on the complex dimension of spaces of cusp forms to show the candidate $f$ we find whose Fourier expansion enumerates ramified covers, cannot be invariant under a congruence subgroup. We also show (Proposition \ref{rational_point}) that Mukamel's cusp classification \cite{mukamel2021polynomials} guarantees the existence of $\mathbb{Q}$-rational cusps on the curve $W_{64}$ we study.

\section{Preliminaries}
\label{Veech}
\subsection{Square-tiled surfaces}
There are an abundance of excellent references on foundational material for translation structures and square-tiled surfaces; we will follow \cite{schmithusen2004}.

Let $X$ be a compact Riemann surface, with possibly finitely many points removed. A \textit{translation structure} $\mu$ on $X$ is an atlas such that all transition functions are translations. The pair $X_{\mu}:=(X,\mu)$ is a $\textit{translation surface}$. The group of all orientation-preserving diffeomorphisms of $(X,\mu)$ will be denoted by $\textrm{Aff}^{+}(X_{\mu})$. Lifting the translation structure $X_{\mu}$ to a translation structure $\tilde{X}_{\mu}$ on the topological universal cover of $X$, charts on $\tilde{X}_{\mu}$ define natural holomorphic maps into $\mathbb{C}$. There is a natural group homomorphism
\begin{displaymath}
\mathbf{aff}:\mathrm{Aff}^{+}(\tilde{X}_{\mu}) \rightarrow \mathrm{Aff}^{+}(\mathbb{C})
\end{displaymath}
obtained by specifying the unique affine diffeomorphism of $\mathbb{C}$ compatible with the holomorphic chart maps. Observe that $\mathrm{Aff}^{+}(\mathbb{C}) \cong \mathbb{C} \rtimes \mathrm{GL}^{+}(2,\mathbb{R})$, so there is a projection $\mathrm{proj} : \mathrm{Aff}^{+}(\mathbb{C}) \rightarrow \mathrm{GL}^{+}(2,\mathbb{R})$, and in fact
a well-defined group homomorphism
\begin{displaymath}
\mathbf{der}: \mathrm{Aff}^{+}(X_{\mu}) \rightarrow \mathrm{GL}^{+}(2,\mathbb{R}), \ f \mapsto \mathrm{proj}(\mathbf{aff}(\hat{f}))
\end{displaymath}
where $\hat{f}$ is some lift of $f$ to $\tilde{X}_{\mu}$. The group $\Gamma(X_{\mu}) := \mathbf{der}(\mathrm{Aff}^{+}(X_{\mu})) \subseteq \mathrm{SL}(2,\mathbb{R})$ is the $\textit{Veech}$ group of $X_\mu$ \cite[Section 2.1]{schmithusen2004}.\\

Let $E_0 := \mathbb{C}/\langle1,i\rangle$ be the standard complex elliptic curve (torus) equipped with a translation structure $\mu$ obtained from a non-zero holomorphic 1-form $\omega$ on $E_0$. A \textit{square-tiled surface} is a pair $(C^{\circ},\pi^{*}\mu)$, where $C$ is a surface, $\pi:C \rightarrow E_0$ is a branched cover with ramification only over the origin $O \in E_0$, and $\pi^{*}\mu$ is the pullback of the translation structure on $E_0$ to $C^{\circ}:= C - \pi^{-1}(O)$. Orientation-preserving affine diffeomorphisms of $C^{\circ}$ descend to orientation-preserving diffeomorphisms of $E_0$ with its origin removed \cite[Proposition 2.6]{schmithusen2004}. Explicitly, \cite[Corollary 2.7]{schmithusen2004} we have
\begin{displaymath}
\Gamma(C^{\circ}, \pi^*\mu) = \left\{A \in \mathrm{SL}(2,\mathbb{Z}) \ | A = \mathrm{\mathbf{der}}(\hat{f}) \ \textrm{for some} \ \hat{f} \in \mathrm{Aff}^{+}(\mathbb{H}) \ \textrm{descending to} \ X \ \textrm{via} \ u \right\},
\end{displaymath}
and further, $\Gamma(C^{\circ}, \pi^*\mu)$ is finite-index in $\mathrm{SL}(2,\mathbb{Z})$. When, $\pi: C^{\circ} \rightarrow E^{\circ}$ is a $G$-torsor, we will denote by $\Gamma(C^{\circ},\pi^{*}\mu)_G < \Gamma(C^{\circ},\pi^{*}\mu)$ the subgroup of the Veech group of $(C^{\circ},\pi^{*}\mu)$ consisting of $G$-equivariant affine diffeomorphisms of $(C^{\circ},\pi^{*}\mu)$.



\subsection{Galois closures}

In this note, we will often pass to Galois closures of  covers obtained by removing branch loci of covers $\pi:C \rightarrow E$, where $E$ is some complex elliptic curve, $C$ is a Riemann surface, and $\pi$ is branched only possibly at the origin $O$. In particular, we will be interested in the unramified cover $C^{\circ} \rightarrow E^{\circ}$ obtained by removing the origin $E^{\circ} - \left\{O\right\}$ and divisor in the pre-image of the origin, i.e., $C^{\circ}:= C - \left\{\pi^{-1}(O)\right\}$ and its Galois closure. We review standard material below, what follows is expository.

\begin{definition}[G-torsors]
\label{G-torsor}
Let $X$ and $Y$ be connected schemes, and $G$ a finite group. A $G$-torsor is a finite \'{e}tale morphism $X \rightarrow Y$, along with an isomorphism $G \overset{\sim}{\rightarrow} \textrm{Gal}(k(X)/k(Y))$ from $G$ to the Galois group of the extension of function fields.
\end{definition}

Let $k$ be an algebraically closed subfield of $\mathbb{C}$ and $X$, $Y$ connected, smooth, projective schemes over $k$ of pure dimension $1$. Fix a $k$-point $y_0 \in Y(k)$. By Riemann's existence theorem \cite[Expos\'{e} XII and Expos\'{e} XIII]{Grothendieck2003sga1}, for any finite group $G$ there is a canonical one-to-one correspondence between $G$-torsors (in Definition \ref{G-torsor}, $G$-torsors are connected) $\pi: X \rightarrow Y$ equipped with a choice of lift $x_0 \in \pi^{-1}(y_0)$ and epimorphisms $\pi_1(Y(\mathbb{C}),y_0) \twoheadrightarrow G$. We also recall that changing the choice of lift $x_0$ corresponds to composing the map $\pi_1(Y(\mathbb{C}),y_0) \rightarrow G$ with an inner automorphism of $G$. \\

For $Y$ as in the previous paragraph, let $y_0 \in Y(\mathbb{C})$ be a $\mathbb{C}$-point. We recall that the \textit{monodromy group} $M$ of an \'{e}tale cover $\pi: X \rightarrow Y$ (with induced cover $\overline{\pi}:X(\mathbb{C}) \rightarrow Y(\mathbb{C}))$ is the minimal quotient of $\pi_1(Y(\mathbb{C}),y_0)$ through which the monodromy action of this group on the fiber $\overline{\pi}^{-1}(y_0)$ factors. \footnote{Alternatively, upon identifying $\overline{\pi}^{-1}(y_0)$ with the symmetric group $S_n$, the monodromy group $M$ may be defined as the image of the homomorphism $\pi_1(Y(\mathbb{C}),y_0)$ to $S_n$ given by the monodromy action. It is a transitive subgroup of $S_n$ precisely when $X(\mathbb{C})$ is assumed connected.} One can find maps $X' \rightarrow X \rightarrow Y$, where $X'$ is smooth and irreducible, such that the extension of function fields $k(X')/k(Y)$ is a Galois closure of $k(X)/k(Y)$ \footnote{Here is a brief sketch of a proof valid in the more general case when the morphism $X \rightarrow Y$ is finite, flat and separable. The assumptions on $X$ and $Y$ imply $X$ and $Y$ are normal and irreducible. The morphism has a well-defined degree $d$. Take the fiber product of $X$ with itself $d$ times, remove the fat diagonal from this fiber product, and normalize irreducible components. The desired $X'$ will be an (any) irreducible $k$-component.} Setting $G = \mathrm{Gal}(k(X')/k(Y))$, $X' \rightarrow Y$ is a $G$-torsor, hence a choice of fiber $x'_0$ in the preimage of $y_0$ under the map $X'(\mathbb{C}) \rightarrow Y(\mathbb{C})$ corresponds to an epimorphism $\pi_1(Y(\mathbb{C}),y_0) \twoheadrightarrow G$. In fact, this epimorphism induces an isomorphism of the monodromy group and $G$, $M \overset{\sim}{\rightarrow} G$. Thus, computing a topological monodromy group of the complex points of a finite, \'{e}tale cover is reduced to computing the Galois group of a Galois closure of the induced extension of function fields. \\

Finally, by the exposition in Hartshorne's textbook \cite[Section I.6]{hartshorne1977algebraic}, a dominant morphism $X \rightarrow Y$ of smooth, connected, projective schemes over $\mathbb{C}$ of pure dimension $1$ gives an isomorphism

\begin{displaymath}
\textrm{Aut}(k(X)/k(Y)) \simeq \textrm{Aut}(X/Y)
\end{displaymath}
between $\textrm{Aut}(k(X)/k(Y))$ and the group of automorphisms of $X$ over $Y$. Further, if $\pi:X \rightarrow Y$ is a finite, \'{e}tale (away from a finite set of points) of smooth, connected, projective $\mathbb{C}$-schemes, and $\left\{p_1,...,p_l\right\} \subset Y$ is the branch locus, $V:= Y \setminus \left\{p_1,...,p_l\right\}$, then
\begin{displaymath}
\pi|_{U}:U \rightarrow V
\end{displaymath}
where $U = \pi^{-1}(V)$, is a connected covering map and in fact
\begin{displaymath}
\textrm{Aut}(X/Y) \simeq \textrm{Deck}(U(\mathbb{C})/V(\mathbb{C})).
\end{displaymath}

Further, the deck group acts transitively when $k(X)/k(Y)$ is Galois, so by the discussion at the beginning of the subsection, in such a situation $\textrm{Gal}(k(X)/k(Y))$ is isomorphic to the monodromy group of $\pi|_{U}:U\rightarrow V$. \\

We record a consequence of the above review pertinent to our purposes.

\begin{prop}
\label{prop_Galois}
Let $E$ be a complex elliptic curve, $C$ a compact Riemann surface, and $\pi:C \rightarrow E$ a finite covering map with branching only possible over the origin of  $O$ of $E$. Define $C^{\circ}:= \pi^{-1}(O)$. The monodromy group of the finite \'{e}tale map $C^{\circ} \rightarrow E^{\circ}$, the Galois closure of the extension $\mathbb{C}(C)/\mathbb{C}(E)$, and the deck group $\mathrm{Deck}({C^{\circ}}'(\mathbb{C})/E^{\circ}(\mathbb{C}))$ are all isomorphic, where ${C^{\circ}}'$ is a Galois closure of the cover.
\end{prop}

\begin{remark}
If $\pi:C^{\circ} \rightarrow E^{\circ}$ is a square-tiled surface, ${C^{\circ}}^{'} \rightarrow C^{\circ} \rightarrow E^{\circ}$ is a Galois closure, and $H \leq F_2 \cong \pi_1({E^{\circ}})$ is the subgroup corresponding to (the complex points of ) $C^{\circ} \rightarrow E^{\circ}$, by standard covering theory the (complex points of) this Galois closure corresponds to the subgroup given by the \textit{normal core}
\begin{displaymath}
\mathrm{Core}(H) = \bigcap_{g\in F_2} gHg^{-1}.
\end{displaymath}
In other words, $\mathrm{Deck}({C^{\circ}}'(\mathbb{C})/E^{\circ}(\mathbb{C})) \cong F_2/\mathrm{Core}(H)$.
\end{remark}

\section{Good reduction}
\

\subsection{Moduli stacks and their compactifications}
\label{modulistacks}

We collect, in a streamlined fashion, background material on some of the moduli spaces that will play a role in our work. In Chen's thesis  \cite{chen2018moduli}, he uses a moduli space of $G$-structures on elliptic curves (for the earliest definition of $G$-structures see \cite{DeligneMumford1969}, and see \cite{PikaartDeJong1997} for a treatment for moduli spaces of higher genus curves) an analog of a moduli space elliptic curves with level structure, where the acting group can be nonabelian. For a comprehensive treatment of the content in this subsection pertaining to moduli spaces of $G$-structures on elliptic curves, we refer the reader to sections 2 and 3 of \cite{chen2018moduli}, which we closely follow. For foundational material on algebraic stacks, see \cite{Olsson2016}.

\begin{definition}[Elliptic curves over schemes]
\label{ellcurve}
An elliptic curve $E$ over a scheme $S$ is a smooth, proper morphism $f: E \rightarrow S$ equipped with a section $e:S \rightarrow E$, such that the geometric fibers of $f$ are connected schemes of pure dimension $1$ and arithmetic genus $1$.
\end{definition}

Let $G$ be a finite \'{e}tale group scheme, and $|G|$ the order of $G$. Let $\mathbb{S} := \textrm{Spec} \ \mathbb{Z}[1/|G|]$. Denote by $\underline{\textbf{Sch}}/\mathbb{S}$ the category of schemes over $\mathbb{S}$, and by $\underline{\textbf{Sets}}$ the category of sets. In  \cite[Section 13.1]{Olsson2016}, a category fibered over $\underline{\textbf{Sch}}/\mathbb{S}$ is constructed, whose objects are triples $(S,E,e)$ as above, where $S$ is an $\mathbb{S}$-scheme and $e$ is the $\textit{zero section}$ of $f:E \rightarrow S$. This is a smooth Deligne-Mumford stack over $\mathbb{S}$, the $\textit{moduli stack of elliptic curves over $\mathbb{S}$}$, which we will denote by $\mathcal{M}(1)$. We will denote its coarse moduli space by $M(1)$.

Let $f: E \rightarrow S$ be an elliptic curve over a connected scheme $S$ and $E^{\circ}:= E - e(S)$ where $e$ is the zero section. Let $\mathbb{L}$ be the set of invertible primes on the scheme $S$, $g:S \rightarrow E^{\circ}$ a section, and $s \in S$ a geometric point. We will denote by 
$\pi_1^{\mathbb{L}}(E^{\circ}/S,g,s)$, the relative \'{e}tale fundamental group (see \cite[Section 2.1.5]{chen2018moduli} for application to the context at hand, and \cite[Section \textrm{V}]{GrothendieckSGA1} for the construction), which is a formal cofiltered limit in the category of finite \'{e}tale group schemes over $S$. Let $G$ a finite \'{e}tale group scheme over $S$, where $|G|$ is invertible on $S$. Denote by 
$(\underline{\textbf{Sch}}/S)_{\text{\'{e}t}}$ the \'{e}tale site of the category of schemes over $S$. Grothendieck showed \cite[Expos\'{e} XIII, 4.5.3]{GrothendieckSGA1} that $\pi_1^{\mathbb{L}}(E^{\circ}/S,g,s)$ is independent, up to canonical isomorphism, of the choice of geometric base point $s \in S$. We may thus consider, without ambiguity, the $\mathrm{Hom}$-sheaf over $(\underline{\textbf{Sch}}/S)_{\text{\text{\'{e}t}}}$, $\Hom(\pi_1^{\mathbb{L}}(E^{\circ}/S,g),G)$, and the subsheaf consisting only of surjective homomorphisms, $\Hom^{\mathrm{sur}}(\pi_1^{\mathbb{L}}(E^{\circ}/S,g),G)$.

On the other hand, $\pi_1^{\mathbb{L}}(E^{\circ}/S,g)$ depends on the choice of section $g$. Nevertheless, there is a natural action of $G$ on $\Hom^{\mathrm{sur}}(\pi_1^{\mathbb{L}}(E^{\circ}/S,g),G)$ by the inner automorphism of group of $G$, which we denote by $\mathrm{Inn}(G)$, and Chen shows \cite[Proposition 2.2.2]{chen2018moduli}, the quotient sheaf $\Hom^{\mathrm{sur}-\mathrm{ext}}(\pi_1(E^{\circ}/S), G):= \Hom(\pi_1^{\mathbb{L}}(E^{\circ}/S),g,G)/\mathrm{Inn}(G)$ is independent, up to canonical isomorphism, of the choice of section $g$. Thus, one can define, without ambiguity 
\begin{displaymath}
\Hom^{\mathrm{sur}-\mathrm{ext}}(\pi_1(E^{\circ}/S), G):= \Hom(\pi_1^{\mathbb{L}}(E^{\circ}/S),g,G)/\mathrm{Inn}(G).
\end{displaymath}

Further, again appealing to Expos\'{e} XIII, 4.5.3 of \cite{GrothendieckSGA1}, Chen defines a sheaf over the \'{e}tale site $(\underline{\textbf{Sch}}/S)_{\text{\'{e}t}}$, from which he can define a $G$-structure on an elliptic curve.

\begin{definition}[$G$-structures, c.f. {\cite[Definition 2.2.3]{chen2018moduli}}]
\label{G-structure}
Let $E$ be an elliptic curve with zero section $e$, and $E^{\circ}:= E-e(S)$. The following assignment defines a presheaf
\begin{displaymath}
(T\rightarrow S) \mapsto \begin{cases}
			\Hom^{\mathrm{sur}-\mathrm{ext}}(\pi_1(E^{\circ}/S),G), & \text{if $E^{\circ}/T$ admits a section $g$ over $T$}\\
            \emptyset, & \text{otherwise}.
		 \end{cases}
\end{displaymath}
We will use the notation $\Hom^{\mathrm{sur}-\mathrm{ext}}(\pi_1(E^{\circ}/S),G)$ for the sheafification of the presheaf defined by this assignment, and define a $G$-structure on $E/S$ to be a global section of this sheaf.
\end{definition}

Finally, can define a functor 

\begin{displaymath}
\mathcal{T}_G: \mathcal{M}(1)_{\mathbb{Z}[1/|G|]} \rightarrow \underline{\textbf{Sets}}
\end{displaymath}
which associates to an object $(S,E,e)$, set of global sections of the sheaf $\Hom^{\mathrm{sur}-\mathrm{ext}}(\pi_1(E^{\circ}/S),G)$ of Definition \ref{G-structure}. This assignment defines on a sheaf on $\mathcal{M}(1)_{\mathbb{Z}[1/|G|]}$ (\cite[Proposition 3.1.2]{chen2018moduli}).
\\

Let $E/S$ be an object in $\mathcal{M}(1)_{\mathbb{Z}[1/|G|]}$. A $G$-structure on $E/S$ is then an element of $\mathcal{T}_G(E/S)$. Chen defines the category  $\mathcal{M}(G)$, fibered over $\underline{\textbf{Sch}}/\mathbb{Z}[1/|G|]$,  to be the category whose objects are pairs $(E/S,\alpha)$, where $E/S$ is an object in $\mathcal{M}(1)$ and $\alpha \in \mathcal{T}_G$ is a $G$-structure. A morphism in $\mathcal{M}(G)$ is a morphism in $\mathcal{M}(1)$, respecting the $G$-structures on the objects in $\mathcal{M}(1)$ (see \cite[Definition 3.1.3]{chen2018moduli} for the formal definition of morphisms).

From Chen's construction we have sketched in this subsection, it is almost immediate that

\begin{theorem}[\cite{chen2018moduli}, Proposition 3.1.4]
\label{etale}
The forgetful map $\mathfrak{f}:\mathcal{M}(G) \rightarrow \mathcal{M}(1)$ sending $(E/S,\alpha)$ to $E/S$ is finite \'{e}tale. Thus, $\mathcal{M}(G)$ is a smooth Deligne-Mumford stack over $\mathrm{Spec} \ \mathbb{Z}[1/|G|]$.
\end{theorem}

For our purposes, we will need a smooth compactification of $\mathcal{M}(G)$ over $\mathbb{Z}[1/|G|]$. The compactification we will use is, briefly, a moduli space of admissible $G$-covers of stable, $1$-pointed elliptic curves in the sense of Kontsevich \cite{manin1999frobenius}. We refer \cite[Section {2}]{chen2024nonabelian} for more details, and record

\begin{prop}[\cite{chen2024nonabelian}, Proposition 2.5.10]
\label{compactification}
    There is a smooth, proper, Deligne-Mumford stack $\overline{\mathcal{M}(G)}$, containing $\mathcal{M}(G)$ as an open, dense substack.
\end{prop}

\subsection{Teichm\"{u}ller uniformization}
\label{TM_section}

In this subsection, we relate the constructions in Section \ref{Veech} to Subsection \ref{modulistacks}, by Teichm{\"u}ller uniformization, and the orbifold quotient (quotient stack) arising from a Veech group of a square-tiled surface. We refer to \cite{perez2024orbifolds} for an in-depth assembly of constructions and results pertaining to framings of elliptic curves, and modular curves from the viewpoint of quotient stacks. We will appeal to \cite{Noohi2004} to for the relationship between a stack constructed as a functor on an analytic category, and the stack constructed as a functor on an algebraic category.  \\

Our goal is to demonstrate that \textit{Teichm\"{u}ller uniformization furnishes an isomorphism between the quotient stack of a modular curve uniformized by a Veech group, and a component of $\mathcal{M}(G)_{\mathbb{C}}$}. This already appears in Chen's survey paper \cite{Chen2025NoncongruenceHurwitz}, with all key ideas present.
\\



By \cite[Expos\'{e} XIII, Corollary 2.12]{GrothendieckSGA1}, if $E$ is an elliptic curve over $\mathbb{C}$, there are canonical isomorphisms $\mathrm{Hom}(\pi_1^{\mathbb{L}}(E^{\circ}),G) \cong \mathrm{Hom}(\pi_1(E^{\circ}),G) \cong \mathrm{Hom}(F_2,G)$. Further, again since $E$ is an elliptic curve over $\mathbb{C}$, by \cite[Proposition 2.2.6]{chen2018moduli}, there are bijections
\begin{displaymath}
\left\{\textrm{$G$-structures on $E$}\right\} \overset{\sim}{\rightarrow} \left\{\textrm{$G$-torsors on $E^{\circ}$ up to isomorphism}\right\} \overset{\sim}{\rightarrow} \mathrm{Hom}^{\mathrm{sur-ext}}(F_2,G).
\end{displaymath}

By the construction of $\mathcal{M}(G)$, the fiber of the forgetful $\frak{f}$ map of Theorem \ref{etale}, over an elliptic curve $E \in \mathcal{M}(1)_{\mathbb{C}}$ may be identified with the set of $G$-structures on $E$ and therefore with the set of $G$-torsors on $E^{\circ}$ (up to isomorphism of the $G$-torsors). 

For $E$ be an elliptic curve over $\mathbb{C}$, recall \cite[Definition 1.9]{perez2024orbifolds} that a \textit{framing} of $E$ is a choice of basis $([e_1],[e_2])$ of the free $\mathbb{Z}$-module $H_1(E(\mathbb{C}),\mathbb{Z}) \cong \mathbb{Z}^2$ such that for any non-zero section $\omega \in H^{0}(E,{\Omega_{E}^1})$, $\mathrm{Im}\frac{\int_{e_1}\omega}{\int_{e_2}\omega} > 0$. Now, assume $\pi: C \rightarrow E_0$ is a branched cover of the standard elliptic curve, $\mu$ is a translation structure on $E_0$ and $(C^{\circ},\pi^{*}\mu)$ is square-tiled. Let $\mathcal{T}(S)$ be the Teichm\"{u}ller space of the topological surface $S$ underlying $E_0$. There is an isomorphism
\begin{equation}
\label{GL(2,R)}
\mathbb{C}^{\times} \backslash \mathrm{GL}^{+}(2,\mathbb{R}) \overset{\sim}{\rightarrow} \mathcal{T}(S)
\end{equation}
induced by deforming the complex structure on $S$ obtained from the translation structure on $E_0$. If $\pi$ is a $G$-torsor, we may consider the connected component $\mathcal{H}(\pi)$ of $\mathcal{M}(G)_{\mathbb{C}}$ containing $\pi$. There is then a natural map $\mathcal{T}(S) \rightarrow \mathcal{H}(\pi)$ obtained as follows: choose an isomorphism $f:\mathbb{Z}^2 \overset{\sim}{\rightarrow} H_1(S,\mathbb{Z})$ such that the intersection number is $+1$ on the canonical basis of $\mathbb{Z}^2$; $f[(1,0)] \cap f[(0,1)] = +1$. Then, as intersection number is preserved by homeomorphism, for a marked complex elliptic curve $(E,m)$ in $\mathcal{T}(S)$, composing $f$ with the marking $m$ of $E$ gives a framing of $E$. Since the hyperbolic plane $\mathbb{H}$ is a fine moduli space for the space of framed complex elliptic curves, we have a map
\begin{equation}
\label{u_T}
u_{\mathcal{T}}: \mathcal{T}(S) \overset{\sim}\rightarrow \mathbb{H} \rightarrow \mathcal{M}(1)_{\mathbb{C}}.
\end{equation}
Now, let $T$ be the underlying topological space of $C$ and continue to denote the topological cover $T \rightarrow S$ underlying $C \rightarrow E_0$ by $\pi$. The map $u_{\mathcal{T}}$ (\ref{u_T}) factors through $u_{\pi}:\mathcal{T}(S) \rightarrow \mathcal{H}(\pi)$ where $u_{\pi}(E,m):= m\circ \pi$. As noted above, isomorphism classes of $G$-torsors on punctured elliptic curves over $\mathbb{C}$ can be identified with elements of $\mathrm{Hom}^{\textrm{sur-ext}}(F_2,G)$, and in fact $u_{\pi}$ induces an isomorphism of stacks
\begin{displaymath}
\mathcal{H}(\pi) \cong [\mathbb{H}/\Gamma_{\pi}]
\end{displaymath}
where $\pi$ is viewed as an element of $\mathrm{Hom}^{\textrm{sur-ext}}(F_2,G)$, and $\Gamma_{\pi}$ is the $\mathrm{Out}^{+}(F_2) \cong \mathrm{SL}(2,\mathbb{Z})$-stabilizer of $\pi$. By inspecting the constructions of the maps (\ref{GL(2,R)}) and (\ref{u_T}), one readily checks \\





\begin{prop}
\label{well-defined}
The Teichm\"{u}ller uniformization $\mathcal{T}(S) \rightarrow \mathcal{H}(\pi)$ induces a well-defined map 
\begin{equation}
\label{well-defined}
[\mathbb{C}^{\times} \backslash \mathrm{GL}^{+}(2,\mathbb{R}) / \Gamma(C^{\circ}, \pi^* \mu)_G] \rightarrow \mathcal{H}(\pi).
\end{equation}
\end{prop}


In fact,

\begin{lemma}[c.f. Proposition 4.13 \cite{Chen2025NoncongruenceHurwitz}]
\label{stackequivalence}

The Teichm\"{u}ller uniformization $\mathcal{T}(S) \rightarrow \mathcal{H}(\pi)$ induces an isomorphism of (complex-analytic) stacks
\begin{equation}
\label{iso_stacks}
[\mathbb{C}^{\times} \backslash \mathrm{GL}^{+}(2,\mathbb{R}) / \Gamma(C^{\circ}, \pi^* \mu)_G] \overset{\sim}{\rightarrow} \mathcal{H}(\pi).
\end{equation}
\end{lemma}
To minimize disruption towards our main objectives, we defer the proof of Lemma \ref{iso_stacks} to Appendix \ref{Appendix_A}.

\begin{remark}
As mentioned, the key geometric fact underlying Lemma \ref{iso_stacks} is \cite[Corollary 2.7]{schmithusen2004}. It is straightforward to prove $G$-equivariant analogs of other results contained therein. For example, an equivariant analog of \cite[Lemma 2.8 (2)]{schmithusen2004} is as follows. Let $\pi:C^{\circ} \rightarrow E^{\circ}$, be a square-tiled surface that is also a $G$-torsor, with corresponding subgroup \(H \subset F_2\). If \(\varphi\) is a \(G\)-equivariant affine automorphism of \(C^{\circ}\), and \(\widehat{\varphi}\) is a lift to the universal cover, then

$$
\widehat{\varphi}\,H\,\widehat{\varphi}^{-1}=H.
$$
Equivalently, the automorphism of \(F_2\) induced by \(\widehat{\varphi}\) preserves the subgroup \(H\):

$$
\widehat{\varphi}_*(H)=H.
$$

\end{remark}








.

\section{Total noncongruence}
\label{total_noncongruence}

Let  $\Gamma < \mathrm{SL}(2,\mathbb{Z})$ be a finite-index subgroup, $l$ its Wohlfahrt level, and $e_l$ its level index. We will say \cite[Definition 3.6]{WeitzeSchmithuesen2015}  $\Gamma$ is totally noncongruence if $e_{l}=1$. By \cite[Corollary 1.1]{WeitzeSchmithuesen2015}, if $\Gamma$ is totally noncongruence, the natural map
\begin{displaymath}
\Gamma \rightarrow \mathrm{SL}(2,\mathbb{Z}/N\mathbb{Z})
\end{displaymath}
is surjective for every natural number. Further, by the discussions in Sections $2$ and $3$ in \cite{chen2026noncongruence}, if $\Gamma < \mathrm{SL}(2,\mathbb{Z})$ is proper and dense in $\mathrm{SL}(2,\hat{\mathbb{Z}})$, then it is totally noncongruence. Further, if $\overline{\Gamma}$ denotes the closure of $\Gamma$ inside $\mathrm{SL}(2,\hat{\mathbb{Z}})$, then we have
\begin{displaymath}
\overline{\Gamma} \cap \mathrm{SL}(2,\mathbb{Z}) = \Gamma^{c}
\end{displaymath}
where $\Gamma^c$ is the smallest congruence subgroup of $\mathrm{SL}(2,\mathbb{Z})$ containing $\Gamma$. Denote by $H(2)$ the moduli space of abelian differentials with a single double zero. The following is the main result of \cite{WeitzeSchmithuesen2015}.

\begin{theorem}[Theorem 1.3, \cite{WeitzeSchmithuesen2015}]
\label{Wohlfahrt}
Let $\Gamma$ be the Veech group of any translation surface in ${H}(2)$. Then either $\Gamma$ is totally noncongruence, or if $l$ the Wohlfahrt level of $\Gamma$, the index of the image of $\Gamma$ in $\mathrm{SL}(2,\mathbb{Z}/l\mathbb{Z})$ is $3$.
\end{theorem}

\begin{lemma}
\label{congruence_closure}
Let $\Gamma$ be the Veech group of any translation surface in ${H}(2)$. Then, the congruence closure $\Gamma^{c}$ is either $\mathrm{SL}(2,\mathbb{Z})$ or an index-$3$ subgroup of $\mathrm{SL}(2,\mathbb{Z})$.
\end{lemma}

Let $\Gamma$ be a finite-index subgroup of $\mathrm{SL}(2,\mathbb{Z})$. A \textit{modular form} for $\Gamma$ of weight $2k$ is a holomorphic function $f$ on $\mathbb{H}$ satisfying, for all $\gamma = \begin{pmatrix} a & b \\
c & d
\end{pmatrix}$,
\begin{enumerate}
\item $f(\gamma z) = (cz + d)^{2k} \cdot f(z)$, for all $z \in \mathbb{H}$ (automorphy);
\item the function $(cz+d)^{-2k}f(\gamma z)$ is bounded for $\mathrm{Im}(z) \rightarrow \infty$ (growth condition).
\end{enumerate}
If in addition, for any $\gamma \in \Gamma$, we have $(cz+d)^{-2k}f(\gamma(z)) \rightarrow 0$ as $\mathrm{Im}(z) \rightarrow \infty$, we say $f$ is a $\textit{cusp form}$. We will denote the space of weight $2k$ modular forms by $M_{2k}(\Gamma,\mathbb{C})$ and the space of weight $2k$ cusp forms by $S_{2k}(\Gamma,\mathbb{C})$. It is a basic fact that both these spaces are finite-dimensional $\mathbb{C}$-vector spaces. There is a positive-definite Hermitian form on $S_{2k}(\Gamma,\mathbb{C})$, given by, for $f(z),g(z) \in S_{2k}(\Gamma,\mathbb{C})$

\begin{displaymath}
\langle f,g\rangle_{\Gamma}
=
\int_{\Gamma\backslash\mathbb{H}}
f(z)\overline{g(z)}\,(\operatorname{Im}z)^k
\frac{dx\,dy}{y^2}
\end{displaymath}
the $\textit{Petersson product}$. For $\Gamma$ noncongruence, let $\Gamma^{c} \supset \Gamma$ be the congruence closure of $\Gamma$ as defined above. Since $S_{2k}(\Gamma^{c}) \subseteq S_{2k}(\Gamma)$, we may take the orthogonal complement $S_{2k}(\Gamma^{c})^{\perp}$ with respect to the Petersson product and obtain the orthogonal decomposition
\begin{equation}
\label{orthogonal_decomposition}
S_{2k}(\Gamma) = S_{2k}(\Gamma^{c}) \oplus S_{2k}(\Gamma^c)^{\perp}
\end{equation}
and will refer to $S_{2k}(\Gamma^{c})^{\perp}$ as the space of \textit{genuine noncongruence cusp forms} for $\Gamma$. By the Eichler-Shimura isomorphism, if we denote the compactified modular curve for $\Gamma$ by $X_{\Gamma}$ (compactification of coarse scheme for $[\mathbb{H}/\Gamma]$), there is an isomorphism $S_2(\Gamma) \overset{\sim}{\rightarrow} H^{0}(X_{\Gamma},\Omega^1_{X_\Gamma/\mathbb{C}})$, where $\Omega^1_{X_{\Gamma}/\mathbb{C}}$ is the sheaf of holomorphic $1$-forms on $X_{\Gamma}$. \footnote{We remind the reader that when the $\Gamma$-action is torsion-free, Kodaira-Spencer theory gives an isomorphism $\omega^{\otimes2} \cong \Omega^{1}_{X_{\Gamma}/\mathbb{C}}(D)$  where $\omega$ is a  line bundle on $X_{\Gamma}$ and $D$ is the cuspidal divisor. The space of sections of $\omega^{\otimes 2k}$ is isomorphic to $M_{2k}(\Gamma,\mathbb{C})$. Further, by Eichler-Shimura theory, $S_{2k}(\Gamma) \cong H^{0}(X_{\Gamma}, \omega^{\otimes 2k-2} \otimes \Omega^1_{X_{\Gamma}/\mathbb{C}})$ can be identified with a parabolic cohomology group of $\Gamma$. We do not have anything to say about higher weight cusp forms at present.}

\begin{lemma}
\label{existence_noncongruence}
Let $\Gamma$ be the Veech group of any translation surface in $H(2)$, and $f$ a weight-2 modular cusp form for $\Gamma$. Then, $f$ is a genuine noncongruence cusp form.
\end{lemma}
\begin{proof}
By Lemma \ref{Wohlfahrt}, there are two cases to consider; either $\Gamma$ is totally noncongruence or if $l$ is the Wohlfahrt level of $\Gamma$, the index of $\Gamma$ in $\mathrm{SL}(2,\mathbb{Z}/l\mathbb{Z})$ is $3$. \\

\textbf{Case (a)} (totally noncongruence) If $\Gamma$ is totally noncongruence, then $\Gamma^{c} = \mathrm{SL}(2,\mathbb{Z})$, but by Riemann-Roch and the identification of weight $2$ cusp forms and holomorphic $1$-forms above
\begin{displaymath}
\mathrm{dim}_{\mathbb{C}} S_2(\mathrm{SL}(2,\mathbb{Z}) = 0
\end{displaymath}
and by $(\ref{orthogonal_decomposition})$, $f \in S_2(\Gamma^{c})^{\perp}$.

\textbf{Case (b)} We first claim that if $\Gamma$ is not totally noncongruence, the index $[\mathrm{SL}(2,\mathbb{Z}):\Gamma^{c}] = 3$. Let $\rho_\ell:\mathrm{SL}(2,\mathbb{Z})\to
\mathrm{SL}(2,\mathbb{Z}/\ell\mathbb{Z})$ be reduction modulo $\ell$, and let
$H=\rho_\ell(\Gamma)$. Set $C=\rho_\ell^{-1}(H)$. Then
$$
[\mathrm{SL}(2,\mathbb{Z}):C]
=
[\mathrm{SL}(2,\mathbb{Z}/\ell\mathbb{Z}):H]
=3.
$$ We claim that $C$ is the congruence closure of $\Gamma$. Indeed, let $D$ be any
congruence subgroup with $\Gamma\subseteq D\subseteq C$. Since $\ell$ is the
Wohlfahrt level of $\Gamma$, $\Gamma(\ell)\subseteq D$. Let $g \in C$. By definition of $C$, $\rho_l(g) \in \rho_l(\Gamma)$ so there exists $\gamma \in \Gamma$ such that $\rho_l(g) = \rho_l(\gamma)$, and therefore $g\gamma^{-1} \in \Gamma(l)$, whence $g \in D$. Thus $C$ is the congruence closure of $\Gamma$, and hence $[\mathrm{SL}(2,\mathbb{Z}):\Gamma^{\mathrm{c}}]=3$. On the other hand, the unique index-$3$ subgroup in $\mathrm{SL}(2,\mathbb{Z})$ is
\begin{displaymath}
\Gamma_0(2)
=
\left\{
\begin{pmatrix}
a & b\\
c & d
\end{pmatrix}
\in \mathrm{SL}(2,\mathbb{Z})
:\ c \equiv 0 \pmod{2}
\right\}.
\end{displaymath}
By the Riemann-Hurwitz formula, the compactified quotient of $\mathbb{H}$ by this subgroup has genus $0$, so again $\mathrm{dim}_{\mathbb{C}}S_2(\Gamma^{c}) = 0$, and $f \in S_2(\Gamma^{c})^{\perp}$.

\end{proof}

We now recall a few arithmetic properties of Teichm\"{u}ller curves in the genus $2$ stratum of Abelian differentials with a single double zero, which we had denoted by $H(2)$, and weight $0$ modular forms (modular functions) on the corresponding quotient $\mathbb{H}/\Gamma$, $\Gamma$ the Veech group of the Teichm\"{u}ller curve.  We will, as above, restrict to those with Veech group a finite index subgroup of $\mathrm{SL}(2,\mathbb{Z})$. In \cite{mukamel2021polynomials}, Mukamel studies natural arithmetic models for such Teichm\"{u}ller curves in $H(2)$. These modular curves have also been intensively studied from points of view beyond Theorem \ref{Wohlfahrt}; for many more details on their geometric construction, basic properties, relationship to Teichm\"{u}ller dynamics and to Hilbert modular surfaces, see \cite{McMullen2005Teichmuller}. We summarize what we need below.
\\

Recall from the introduction that discriminants $D = d^2$ of quadratic orders $\mathcal{O}_{d^2}$ (in the split quadratic algebra over $\mathbb{Q}$), there is any infinite family of algebraic curves indexed by $d \geq 3$, $W_{d^2}$, each lying in the Hilbert modular surface $X_{d^2}$, a surface which parametrizes principally polarized abelian surfaces with multiplication by $\mathcal{O}_{d^2}$. Each curve $W_{d^2}$ parametrizes genus-two curves admitting a primitive degree-$d$ map to an elliptic curve with a single critical point. Each irreducible component $W \subset W_{d^2}$ is uniformized by a finite index subgroup of $\mathrm{SL}(2,\mathbb{Z})$, $\Gamma$, and the algebraic immersion $\mathbb{H}/\Gamma \rightarrow \mathcal{M}_2$ is a local isometry for the Kobayashi metric on the moduli space of genus $2$ surfaces $\mathcal{M}_2$. In other words, the image is a \textit{Teichm\"{u}ller curve}. By the realization of $W_{d^2}$ as a moduli problem, if ($\phi: Y \to E$) is such a genus $2$ cover of an elliptic curve, the induced map on Jacobians has kernel an elliptic curve $\mathrm{ker}_*(\phi)$, thus there is a natural map

\begin{displaymath}
W\longrightarrow \mathcal M_1\times\mathcal M_1,\qquad
(Y,\phi)\longmapsto (E,\mathrm{ker}_*(\phi)).
\end{displaymath}
Using the elliptic $j$-invariant, we then have a map

\begin{displaymath}
\psi: W \longrightarrow \mathbb C^2,\qquad (\phi:Y \rightarrow E) \longmapsto (j(E),j(\mathrm{ker}_{*}(\phi)).
\end{displaymath}
Letting $j_1$, $j_2$, respectively, be the $j$-invariant of the base elliptic curve, and the $j$-invariant of the kernel of $\phi$, the \textit{image} of the above map, $\psi(W)$, is an algebraic curve in the $(j_1,j_2)$-plane. Mukamel constructs an arithmetic model of this curve as a zero locus in $\operatorname{Spec}\mathbb Z[j_1,j_2]$, obtaining a scheme $V_W$ whose generic characteristic-zero fiber is $\psi(W)$. Finally, let $\overline{W}$ denote the smooth, projective curve birational to $W$. By \cite[Proposition 2.5]{mukamel2021polynomials}, the curve $V_W$ is shown to be birational to $W$ by extending $\psi: W \rightarrow \mathbb{C}^2$ to a map $\overline{W} \rightarrow \mathbb{P}^1 \times \mathbb{P}^1$, and using his classification of cusps \cite[Propositions A.7 and A.8]{mukamel2021polynomials}.

\begin{prop}
\label{rational_point}
If $d \geq 3$ and $d \equiv 0 \ \mathrm{mod} \ 4$ or $d \equiv 2 \ \mathrm{mod} \  4$, $\overline{W}_{d^2}$ has a $\mathbb{Q}$-rational cusp.
\end{prop}

\begin{proof}

First, by the classification in \cite{McMullen2005Teichmuller}, $W_{d^2}$ is irreducible, so in the subsequent proof we set $W = W_{d^2}$. Let $C\subset \mathbf P^1_{\mathbf Q}\times\mathbf P^1_{\mathbf Q}$ be the projective closure of the affine curve $\psi(W)$. By \cite[Proposition 2.5]{mukamel2021polynomials}, $C$ is birational to $W$. Recall we denoted by $\overline{W}$ the smooth projective model of $W$, and we now let $\widetilde C$ denote the normalization of $C$. The birational map of \cite[Proposition 2.5]{mukamel2021polynomials} induces an isomorphism of function fields $\mathbf Q(\overline{W})\cong\mathbf Q(\widetilde C)$, hence an isomorphism $\overline{W} \cong\widetilde C$. Consider the $\mathbf Q$-rational point $P=(\infty,\infty)\in\mathbf P^1\times\mathbf P^1$. A branch of $C$ at $P$ is simply a point of the normalization $\widetilde C$ lying above $P$. Since $\widetilde C\cong \overline{W}$, these branches are equivalently the points of $\overline{W}$ at which both $j_1$ and $j_2$ have poles. By the discussion in the \cite[Appendix A]{mukamel2021polynomials}, these branches are identified with the two-cylinder cusps of $W_{d^2}$. Thus the two-cylinder cusps are precisely the points $w \in \overline{W}(\overline{\mathbf Q})$ satisfying $\psi(w)=(\infty,\infty)$, where we mildly abuse notation by denoting the extended map also by $\psi:\overline{W}\to\mathbf P^1\times\mathbf P^1$. By the cusp classification \cite[Propositions A.7 and A.8]{mukamel2021polynomials}, there is a unique two-cylinder cusp. Denote the corresponding point of $\overline{W}(\overline{\mathbf Q})$ by $w$ .We now use the $\mathbf Q$-structure. The morphism $\psi$ is defined over $\mathbf Q$, and $P=(\infty,\infty)$ is a $\mathbf Q$-rational point. Hence, for every $\sigma\in\operatorname{Gal}(\overline{\mathbf Q}/\mathbf Q)$, we have $\psi(\sigma w)=\sigma(\psi(w))=\sigma(P)=P$. Thus $\sigma w$ is another geometric point of $\overline{W}$ lying above $P$. Moreover, the branch corresponding to $\sigma w$ has the same local pole-order data as the branch corresponding to $w$. Indeed, because $j_1,j_2\in\mathbf Q(\overline{W})$, the Galois action preserves their valuations: $\operatorname{ord}_{\sigma x}(j_i)=\operatorname{ord}_x(j_i)$ for $i=1,2$ (in the notation of \cite[Appendix A]{mukamel2021polynomials}, this means that $\sigma w$ has the same pair $(v_1,v_2)$ as $w$). Hence $\sigma w$ is again the two-cylinder cusp corresponding to the same branch type. Since the two-cylinder cusp is unique, it follows that $\sigma w=w$ for every $\sigma\in\operatorname{Gal}(\overline{\mathbf Q}/\mathbf Q)$. Therefore the residue field $\kappa(w)$ is fixed by $\operatorname{Gal}(\overline{\mathbf Q}/\mathbf Q)$. Equivalently, $\kappa(w)=\mathbf Q$, so $w\in \overline{W}(\mathbf Q)$. Thus the unique two-cylinder cusp of $W_{d^2}$ is $\mathbf Q$-rational.

\end{proof}

\section{Main theorem proofs}

We remind the reader of our notation in Section \ref{Wohlfahrt}; for $\Gamma \in \mathrm{SL}(2,\mathbb{Z})$ finite index, $X_{\Gamma}$ is the compactification of the coarse scheme of $[\mathbb{H}/\Gamma]$. When $f \in M_k(\Gamma,\mathbb{C})$ and $c$ is a cusp of $\Gamma$ width $h$, we say $f = \sum\limits_{n \geq 0}a(n)q^{n/h}$ is a $q$-expansion of $f$ at $c$ if $q = e^{2\pi i \gamma(z)}$ and $\gamma(c) = \infty$, $z$ is a local coordinate at $i\infty$. The following is originally due to Atkin Swinnerton-Dyer \cite{atkin1971modular}, with at least three other different perspectives on the this congruence, due to Katz \cite{katz1981crystalline}, Scholl \cite{Scholl1985} \footnote{Theorem 5.4 in \cite{Scholl1985} gives a more general result. In particular, a consequence of Theorem 5.4 and modularity is the following. Suppose $X_\Gamma$ has a $\mathbb{Q}$-model, a $\mathbb{Q}$-rational point at $\infty$, $k\geq 2$, $S_k(\Gamma)$ is one-dimensional, and $f$ is nonzero and has rational Fourier coefficients $a(n)$. There is a normalized newform $g$ of weight $k$, level $N$ and real Dirichlet character $\chi$ such that for good primes $p$ the $p$-th Hecke eigenvalues of $g$, $b(p)$, and $a(n)$ satisfy $a(np) -b(p)a(n) + \chi(p)p^{k-1}a(n/p) \equiv 0 \ \mathrm{mod} \ p^{(k-1)(1 + \mathrm{ord}_p(n))}$. } and Kibelbek \cite{Kibelbek2014ASD}.



\begin{theorem}[c.f. \cite{atkin1971modular}, \cite{Scholl1985},
\cite{katz1981crystalline}, 
\cite{Kibelbek2014ASD}]
\label{katz-scholltheorem}
Let $\Gamma$ be a finite-index subgroup of \ $\mathrm{SL}(2,\mathbb{Z})$, and assume the associated modular curve $X_{\Gamma}$ is genus $1$, defined over $\mathbb{Q}$, and has a $\mathbb{Q}$-rational point at the cusp $\infty$. There is a weight $2$ cusp form on $X_{\Gamma} \simeq E$, such that the (normalized) holomorphic differential $1$-form on the elliptic curve $E$, $f \frac{dq}{q} = \sum\limits_{n \geq 1}a(n)q^n\frac{dq}{q}$ satisfies the congruence relation

\begin{equation}
\label{Katz-Scholl}
a(np^r) - [p+1 - \#E(\mathbb{F}_p)]a(np^{r-1}) + pa(np^{r-2}) \equiv 0 \ \mathrm{mod} \ p^r
\end{equation}
for all primes $p$ not dividing the conductor of $E$, and $n \geq 1$.
\end{theorem}

\begin{theorem}[c.f. \cite{selberg1965}]
Let $f(z) = \sum\limits_{n \geq 1}a(n) q^{n/h}$ be the Fourier expansion of a weight $k$ cusp form at a cusp of width $h$, for any finite-index subgroup of $\mathrm{SL}(2,\mathbb{Z})$. Then, there exists a constant $C$, depending only on $f$, such that
\begin{equation}
\label{Selberg}
|a(n)| < Cn^{k/2 - 1/5} \ for \ all \ n\geq 1.
\end{equation}
\end{theorem}

We also recall the $\textit{Hasse-Weil bound}$, that for any elliptic curve $E$ over a finite field $\mathbb{F}_{q}$,
\begin{equation}
\label{Hasse_Weil}
|\#E(\mathbb{F}_q) - (q+1)| \leq 2\sqrt{q}.
\end{equation}

\begin{theorem}
\label{useful_theorem}
Let $\overline{\mathcal{H}}(\pi)$ be the coarse scheme of a connected component of $\overline{\mathcal{M}(G)}$. Assume $\overline{\mathcal{H}}(\pi)$ is defined over $\mathbb{Q}$, is of genus $1$ and has a $\mathbb{Q}$-rational cusp at $c$. There exists a weight $2$ cusp form $f$ for the associated modular group of $\overline{\mathcal{H}}(\pi)_{\mathbb{C}}$, whose Fourier coefficients at $c$, $a(n)$, determine  \ $\#\overline{\mathcal{H}}(\pi)(\mathbb{F}_p)$ for all primes $p \nmid |G|$. Furthermore, if $n$ is such that $a(n) \neq 0 \ \mathrm{mod} \ p$,  there exists an integer $r$ (depending on $n$ and $p$) such that $\#\mathcal{H}(\pi)(\mathbb{F}_p)$ is already determined by $a(np^r)$, $a(np^{r-1})$, $a(np^{r-2})$.
\end{theorem}

See \cite{li2012fourier} for a related strategy leveraging Selberg's bounds. \\


\begin{proof}

Let $f \in S_2(\Gamma,\mathbb{C})$ be as in Theorem \ref{katz-scholltheorem}, and let
$$
f(z)=\sum_{m\geq 1}a(m)q^{m/h},
\qquad q=e^{2\pi iz},
$$
be the Fourier expansion of $f$ at the $\mathbb{Q}$-rational cusp, and suppose it is of width $h$.
Suppose that $f\not\equiv 0\pmod p$. Thus there exists some $n\geq1$
such that
$$
a(n)\not\equiv0\pmod p.
$$
In particular, $a(n)\neq0$. Set
$$
A_p=p+1-\#E(\mathbb F_p).
$$
By the Hasse-Weil bound (\ref{Hasse_Weil}),
$$
|A_p|\leq 2p^{1/2}.
$$
By Selberg's bound in weight $2$ (\ref{Selberg}), there is a constant $C>0$ such that
$$
|a(m)|\leq C m^{4/5}
$$
for all $m\geq1$. Fix $n$ as above and set
$$
C_n=Cn^{4/5}.
$$
Then
$$
|a(np^r)|\leq C_n p^{4r/5}.
$$
Choose $R$ sufficiently large
(depending on $n$ and $p$) so that, for every $r\geq R$,
\begin{equation}
\label{fund_inequality}
C_n p^{4r/5}
+
2C_n p^{1/2+4(r-1)/5}
+
C_n p^{1+4(r-2)/5}
<p^r.
\end{equation}
This is possible because
$$
\frac45r<r,
\qquad
\frac12+\frac45(r-1)<r,
\qquad
1+\frac45(r-2)<r.
$$
Now the ASD congruences (\ref{Katz-Scholl}) give
\begin{displaymath}
a(np^r)-A_p a(np^{r-1})
+p\,a(np^{r-2})
\equiv0\pmod{p^r},
\end{displaymath}
where the final term is omitted when $r=1$ by our convention. For $r=1$, we have
\begin{equation}
\label{eq1thm}
a(np)-A_pa(n)\equiv0\pmod p
\end{equation}
and therefore either $a(np) = A(p)a(n)$ or the left hand side of (\ref{eq1thm}) is nonzero and divisible by $p$. In the first case, $A(p)$ is determined, as $a(n) \neq 0$ and we are done. In the second case, by the Hasse-Weil bound (\ref{Hasse_Weil}) and the fact $a(n) \neq 0 \ \mathrm{mod} \ p$, $a(np)$ must be nonzero. Now, assume $r \geq R$ and consider the general form of the ASD congruence (\ref{Katz-Scholl}). By (\ref{fund_inequality}), we have 
$$
\begin{aligned}
&\left|
a(np^r)-A_pa(np^{r-1})
+p\,a(np^{r-2})
\right| \\
&\qquad\leq
C_n p^{4r/5}
+
2C_n p^{1/2+4(r-1)/5}
+
C_n p^{1+4(r-2)/5}
<p^r.
\end{aligned}
$$
Since the expression is divisible by $p^r$, it must therefore vanish:
$$
a(np^r)-A_pa(np^{r-1})
+p\,a(np^{r-2})=0.
$$
By inducting on $r$, with the base case $r=1$ handled as above, we may assume $a(np^{r-1})\neq0$, and obtain
$$
A_p=
\frac{a(np^r)+p\,a(np^{r-2})}
     {a(np^{r-1})}.
$$
Thus
$$
\#E(\mathbb F_p)=p+1-A_p
$$
is determined by
$$
a(np^r),\qquad a(np^{r-1}),\qquad a(np^{r-2})
$$
for $r\geq R$.
\end{proof}

\begin{proof}[proof of Theorem \ref{main_theorem}]
We are reduced, by Theorem \ref{useful_theorem}, to verifying that the associated modular curve for the Veech group of the Galois closure of a degree $d$ cover $\pi:C^{\circ} \rightarrow E^{\circ}$ can be identified with the coarse scheme of a connected component of $\mathcal{M}(G)_{\mathbb{C}}$ for some $G < S_d$. This follows immediately from Lemma \ref{stackequivalence}, upon noting that Galois group of the Galois closure is isomorphic to the monodromy group of $\pi$.
\end{proof}

\begin{proof}[proof of Theorem \ref{maintheorem2}]
By the proof of Theorem \ref{main_theorem}, and since we have already noted $\overline{W}_{64}$ is an elliptic curve over $\mathbb{Q}$ we only need verify 
\begin{enumerate}
\item $\overline{W}_{64}$ has a $\mathbb{Q}$-rational cusp $c$, and
\item the cusp form $f$ obtained from Theorem \ref{katz-scholltheorem} is genuinely noncongruence.
\end{enumerate}
The first claim follows immediately from Proposition \ref{rational_point}, and the second immediately from Lemma \ref{existence_noncongruence}.
\end{proof}

\section{Concluding remarks, questions and conjectures}
As Theorems \ref{main_theorem} and \ref{maintheorem2} are outputs of an interplay between integral models of modular curves, Atkin Swinnerton-Dyer congruences on Fourier coefficients of cusp forms, and crucially, Rankin-Selberg growth bounds on the Fourier coefficients $a(n)$ of $q$-expansions, analogies with congruence modular forms raise the question of establishing sharp bounds on $|a(n)|$. Indeed, Deligne deduces the Ramanujan-Petersson conjecture for the prime $p$ coefficients of a weight $k$ cusp eigenform
\begin{displaymath}
|a(p)| \leq 2p^{({k-1})/2}
\end{displaymath}
from his purity theorem \cite[Th\'eor\`eme 8.2]{PMIHES_1974__43__273_0}. This also relies on the Eichler-Shimura congruence relation, which realizes the $p$-th Hecke operator $T_p$ as the sum of Frobenius and its transpose mod $p$. Proving this congruence relation depends on understanding the reduction mod $p$ of the level $Np$ modular curve $X_0(pN)$ for $p \nmid N$, and the relationship to Fourier coefficients relies on the action of $T_p$ on $q$-expansions. In the general noncongruence case, understanding the mod $p$ reduction at bad primes of the connected components Hurwitz stacks is in general a hard problem, but one which appears approachable in some cases \cite{bouw2004reduction}. On the other hand, there is no reason to expect that there exist any natural correspondences on noncongruence modular curves whose eigenvalues acting on cohomology have any direct relationship to Fourier coefficients of genuine noncongruence forms.  
\\

We believe the following to be a natural question. Let $a(n)$ be the $n$-th Fourier coefficient of the $q$-expansion of a modular cusp form $f$ of weight $k$, modular for a finite index noncongruence subgroup of $\mathrm{SL}(2,\mathbb{Z})$. Assume $f$ is not modular for any congruence subgroup. 
\begin{ques}
\label{ques1}
Does there exist $\epsilon > 0$ such that
${|a(n)|}\big/{n^{\frac{k-1+\epsilon}{2}}}$ is unbounded as $ n \rightarrow \infty$?
\end{ques}

Question \ref{ques1} fits into three different but related frameworks. Calegari-Dimitrov-Tang \cite[Section 7.4.8]{CalegariDimitrovTang_UnboundedDenominators} point out that their resolution of the Unbounded Denominators Conjecture puts Question \ref{ques1} on a common footing with the question of \textit{deficient places} for a weight $k$ cuspidal component function $f$ of a multiplier system associated to a representation $\rho:\mathrm{PSL}(2,\mathbb{Z}) \rightarrow \mathrm{GL}(n,\mathbb{C})$. More specifically, for a place $v \in M_{\mathbb{Q}}$, which may be infinite, we say it is \textit{deficient for $f$} if there exists a Fourier expansion of $f$ at some cusp, with $\mathbb{Q}$-coefficients, such that

\begin{displaymath}
\begin{cases}
{|a(n)|_v}\big/{n^{\frac{k-1+\epsilon}{2}}} \ \text{is unbounded for some} \ \epsilon>0, & \text{if } v = \infty;\\
|a(n)|_v \ \textrm{is unbounded}, & \text{if} \ v \in M^{\textrm{fin}}(\mathbb{Q}).
\end{cases}
\end{displaymath}
A consequence of the solution to the Unbounded Denominators Conjecture is the statement that if $f$ is not invariant under a congruence subgroup, there is a deficiency place for $f$ that is finite. Question \ref{ques1} asks whether $v = \infty$ is a deficiency place if $f$ is invariant under a finite index subgroup, but not invariant under a congruence subgroup. \\

The family of curves $W_{d^2}$ described in Section \ref{total_noncongruence} are not the only Teichm\"{u}ller curves in $H(2)$. For any discriminant $D \geq 5$, and not a square, there is in fact a Teichm\"{u}ller curve $W_D$, and each is also contained in a the Hilbert modular surface $X_D$. The curve $W_D$ is a totally geodesic subvariety of $X_D$ in the Kobayashi metric on $X_D$, and the uniformizing Veech group of $W_D$ is no longer arithmetic as in the case of a square, but has quadratic trace field. Work of Zagier-M\"{o}ller \cite{MoellerZagier2015Modular} shows that this Hilbert modular embedding gives rise to a natural notion of $\textit{twisted modular forms}$ for these nonarithmetic Veech groups. They prove \cite[Theorem 2.1]{MoellerZagier2015Modular} using standard methods of Hecke, that Fourier coefficietns of cuspidal twisted modular forms admitting expansions $f = \sum a(n)q^n$ at a cusp satisfy $a(n) = O(n^{K/2})$ where $K$ is the bi-weight (the natural analog of weight). In the spirit of Question \ref{ques1}, one can ask for optimal bounds on the growth of Fourier coefficients of twisted cusp forms for nonarithmetic Veech groups. \\

Finally, we recall that the Deligne bound for holomorphic cusp forms can be put into yet another framework in view of the Langlands program. In particular, let $G$ be a reductive over a global field $K$ and $\varPi$ a cuspidal, automorphic, irreducible, unitary representation of $G$ over $K$. As a global representation $G(\mathbb{A}_K)$ of its adelic points, $\varPi$ is a restricted product
\begin{displaymath}
\varPi = \otimes'_v \varPi_v
\end{displaymath}
where $v$ ranges over all places of $K$ and $\varPi_v$ are local representations of $G(K_v)$ at the completions $K_v$. A representation $\varPi$ satisfies the \textit{generalized Ramanujan conjecture} if each local constituent $\varPi_v$ is tempered; that is, the matrix coefficients of $\varPi_v$ lie in $L^{2+\epsilon}(G(K_v))$ for all $\epsilon > 0$. The Deligne bound for holomorphic cusp forms implies all local constituents at finite places $p$ are tempered for $G=\mathrm{GL}(2)$ and $K = \mathbb{Q}$. On the other hand, if $\varPi$ comes from a Maass form $f$, the condition of being tempered at $\varPi_{\infty}$ is exactly that $f$ is an eigenfunction of the Laplace operator with eigenvalue $\geq 1/4$. Selberg's $1/4$ conjecture \cite{Selberg1956} that Maass forms on congruence subgroups have first positive Laplace eigenvalues $\lambda_1$ uniformly bounded away from $1/4$ is thus the statement that cuspidal automorphic representations $\varPi$ of $\mathrm{GL}(2,\mathbb{A}_{\mathbb{Q}})$ coming from Maass cusp forms are "tempered at infinity", in the sense that $\varPi_{\infty}$ is tempered. Despite the fact the family of modular curves $W_{d^2}$ are noncongruence (Theorem \ref{Wohlfahrt}), a consequence of a longstanding conjecture of McMullen \cite[Conjecture 1.1]{jeffreys2026euler} is that this family of curves forms an expander family; that is, $\lambda_1$ of the Laplace operators on the family of curves is uniformly bounded away from $0$. We do not know of any overarching theory that can encompass Question \ref{ques1} and McMullen's conjecture. We note, however, that Solan \cite[Theorem 1.2]{Solan2024} recently established a "complementary" result to McMullen's conjecture; that in each stratum of abelian differentials, Zariski dense nonlattice Veech groups have critical exponent uniformly bounded away from $1$.

\appendix
\section{Appendix}
\label{Appendix_A}

We record here the proof of Lemma \ref{stackequivalence}. \\

We give a rapid review of stack morphisms. We use the notation of Section \ref{modulistacks}.  Let $\mathbb{S}$ be a scheme and once again $\underline{\mathrm{\textbf{Sch}}}/\mathbb{S}$ the category of $\mathbb{S}$-schemes. Recall that a category over $S$ is a category $\mathcal{F}$ and a functor
\begin{displaymath}
p:\mathcal{F}\rightarrow \underline{\mathrm{\textbf{Sch}}}/\mathbb{S}
\end{displaymath}
henceforth referred to as the $\textit{projection functor}$. A $\textit{stack}$ over $\mathbb{S}$ is a category $\mathcal{F}$, over $\mathbb{S}$, fibered in groupoids, such that the assignment
\begin{align*}
\underline{\mathrm{\textbf{Sch}}}/\mathbb{S} &\mapsto \mathrm{\textbf{Sets}} \\
U &\mapsto \mathcal{F}(U) := p^{-1}(U)
\end{align*}
is a sheaf of groupoids in the \'{e}tale topology on $\underline{\mathrm{\textbf{Sch}}}/\mathbb{S}$. Let $\mathcal{F}$, $\mathcal{G}$ be two stacks over $\mathbb{S}$, and denote by $p_{\mathcal{F}}$ and $p_{\mathcal{G}}$, respectively, their projection functors. A \textit{morphism of stacks} is a functor 

\begin{displaymath}
\Pi: \mathcal{F} \rightarrow \mathcal{G}
\end{displaymath}
satisfying $p_{\mathcal{F}} = p_{\mathcal{G}} \circ \Pi$.

\begin{definition}[Isomorphism of stacks]
\label{isomorphismstacks}
We say a morphism of stacks $\Pi: \mathcal{F} \rightarrow \mathcal{G}$ is an isomorphism if $\Pi$ is an equivalence of categories.
\end{definition}

Let $\mathcal{A}\!n$ be the category of complex-analytic spaces. An \textit{stack on $\mathcal{A}\!n$} is defined completely analogously to a stack over $\mathbb{S}$, exactly upon replacing the category $\underline{\mathrm{\textbf{Sch}}}/\mathbb{S}$ by $\mathcal{A}\!n$ in the definition above. The definitions of morphisms and isomorphisms \ref{isomorphismstacks} are also completely analogous. By \cite[Theorem 3.11]{perez2024orbifolds}, when $X$ is complex-analytic space, and $G$ a discrete group action by automorphisms of $X$, the orbifold quotient $[X/G]$ is constructed as a stack on $\mathcal{A}\!n$. We remark that in \cite[Definition 2.7]{perez2024orbifolds} the definition is given as a functor from $\mathcal{A}\!n$ to the category of groupoids, but by a standard construction (Grothendieck) is equivalent to our definition. \\

Now, let $\mathcal{H}(\pi) \subset \mathcal{M}(G)_{\mathbb{C}}$ be as in Section \ref{TM_section}. Observe \cite[Remark 3.1.5]{chen2018moduli} that the automorphisms of a $G$-torsor on a complex elliptic curve, $X \in \mathcal{H}(\pi)$ are precisely those automorphisms of the elliptic curve that lift to an automorphism on the cover. On the other hand, by \cite[Proposition 2.6 (3)]{schmithusen2004} \footnote{Strictly speaking, this should be stated for $G$-equivariant affine diffeomorphisms, but the proof works mutatis mutandi.}, the automorphisms of an equivalence class of square-tiled surfaces $[X] \in [\mathbb{C}^{\times} \backslash \mathrm{GL}^{+}(2,\mathbb{R}) / \Gamma(C^{\circ}, \pi^* \mu)_G]$ are precisely those automorphisms of the base flat torus that lift to $G$-equivariant automorphisms of the cover \footnote{If we had considered the coarser relation on $\mathbb{C}^{\times} \backslash \mathrm{GL}^{+}(2,\mathbb{R})$-orbits where we had taken the quotient by $\Gamma(C^{\circ}, \pi^* \mu)$, rather than  $\Gamma(C^{\circ}, \pi^* \mu)_G$, we would identify orbit points differing by an element of $\mathrm{Out}(G)$, thus would need to consider the quotient $\mathcal{M}(G)/\mathrm{Out}(G)$. }. This is the key observation.

\begin{proof}[proof of Lemma \ref{stackequivalence}]
Write $\mathcal{H}(\pi)^{an}$ for the analytification of $\mathcal{H}(\pi)$. This is a stack on  $\mathcal{A}\!n$ and it is evident that the Teichm\"{u}ller uniformization map is a morphism of stacks on $\mathcal{A}\!n$ from the orbifold quotient $[\mathbb{C}^{\times} \backslash \mathrm{GL}^{+}(2,\mathbb{R}) / \Gamma(C^{\circ}, \pi^* \mu)_G]$ to $\mathcal{H}(\pi)^{an}$, and further, that it is essentially surjective. We will denote it by $\Pi$. To verify full faithfulness, we need to show that for any $S \in \mathcal{A}\!n$ a complex-analytic space, and any $X,Y \in[\mathbb{C}^{\times} \backslash \mathrm{GL}^{+}(2,\mathbb{R} / \Gamma(C^{\circ}, \pi^* \mu)_G]$ in the fiber over $S$ we have

\begin{equation}
\label{main_stack}
\mathrm{Hom}_{[\mathbb{C}^{\times} \backslash \mathrm{GL}^{+}(2,\mathbb{R}) / \Gamma(C^{\circ}, \pi^* \mu)_G](S)}(X,Y) \simeq \mathrm{Hom}_{\mathcal{H}(\pi)^{an}(S)}(\Pi(X),\Pi(Y))
\end{equation}
and we remind the reader of the notation we chose that for a stack $\mathcal{F}$ and complex-analytic space $S$, $\mathcal{F}(S)$ denotes the fiber of the projection functor over $S$. Since these fibers are groupoids, to verify (\ref{main_stack}), it suffices to check
\begin{enumerate}
    \item $X \cong Y \iff \Pi(X) \cong \Pi(Y)$, and
    \item $\mathrm{Aut}_{\mathrm{GL}^{+}(2,\mathbb{R}) / \Gamma(C^{\circ}, \pi^* \mu)_G](S)}(X)$ is bijective with $\mathrm{Aut}_{\mathcal{H}(\pi)^{an}(S)}(\Pi(X))$.
\end{enumerate}
The first point is evident, and $2.$ follows immediately from the observation in the paragraph preceding the beginning of this proof on the characterization of automorphisms of $G$-torsors and equivalence classes of square-tiled surfaces up to $G$-equivariant automorphisms.

\end{proof}

\bibliographystyle{plain}
\bibliography{bibliography.bib}

\end{document}